\documentclass[11pt]{amsart}
\usepackage[a4paper,margin=1in]{geometry}
\usepackage{amsmath,amssymb,amsthm,mathtools}
\usepackage{microtype}
\usepackage{enumitem}
\usepackage{booktabs}
\usepackage{float}
\usepackage{tikz}
\usepackage{xurl}
\usepackage[colorlinks=true,citecolor=blue,linkcolor=blue,urlcolor=blue]{hyperref}

\newtheorem{theorem}{Theorem}[section]
\newtheorem{corollary}[theorem]{Corollary}
\newtheorem{lemma}[theorem]{Lemma}
\newtheorem{proposition}[theorem]{Proposition}
\theoremstyle{definition}
\newtheorem{definition}[theorem]{Definition}
\newtheorem{example}[theorem]{Example}

\theoremstyle{remark}
\newtheorem{remark}[theorem]{Remark}

\numberwithin{equation}{section}

\newcommand{\NN}{\mathcal N}
\newcommand{\DD}{\mathcal D}
\newcommand{\Ho}{\mathcal H_{\!o}}
\newcommand{\Cl}{\operatorname{Cl}}
\newcommand{\CN}{\operatorname{CN}}
\newcommand{\pd}{\operatorname{pdim}}
\newcommand{\depth}{\operatorname{depth}}
\newcommand{\reg}{\operatorname{reg}}
\newcommand{\mult}{\operatorname{e}}

\title[Neighborhood complexes with 2-linear Stanley--Reisner ideals]{A Complete Classification of 2-Linear Neighborhood Complexes}

\author{Mohammed Rafiq Namiq}
\address{Department of Mathematics, College of Science, University of Sulaimani, Sulaymaniyah, Kurdistan Region, Iraq}
\email{mohammed.namiq@univsul.edu.iq}
\date{}


\hypersetup{
	pdftitle={A Complete Classification of 2-Linear Neighborhood Complexes},
	pdfauthor={Mohammed Rafiq Namiq},
	pdfsubject={Neighborhood complexes with 2-linear Stanley--Reisner ideals},
	pdfkeywords={
		neighborhood complex,
		Stanley--Reisner ideal,
		linear resolution,
		hypertree,
		graded Betti number,
		cactus graph
	}
}

\begin{document}
	
	\begin{abstract}
		Let \(G\) be a nonempty finite simple graph. We study when the
		Stanley--Reisner ideal of its neighborhood complex has a \(2\)-linear
		resolution. Combining Fr\"oberg's theorem with the classical hypertree criterion, we obtain the following equivalent description in graph terms: \(G\) is bipartite, its 
		indexed open neighborhoods are Helly, and every induced cycle of length at least eight has
		a filling from each color class. This class properly contains the chordal
		bipartite graphs without isolated vertices. Hochster's formula gives all
		squarefree multigraded Betti numbers, while face counts determine the complete
		graded Betti table. If \(G\) has \(n\) vertices and \(c\) connected
		components, then its Stanley--Reisner ring has terminal Betti number
		\(2c-1\), projective dimension \(n-1\), and depth one. We also determine
		the multiplicity and the initially Cohen--Macaulay and Cohen--Macaulay cases.
		A second formula separates degree data from overlaps caused by repeated common
		neighbors and yields closed expressions for bipartite graphs without
		\(K_{2,3}\), cactus graphs, pseudoforests, and forests. For square cactus
		graphs, the Betti table recovers every degree multiplicity at least three; for
		forests, it recovers the complete degree sequence. Finally, the dominance
		complex has a \(2\)-linear Stanley--Reisner ideal precisely for nontrivial
		stars.
	\end{abstract}
	
	\maketitle
	
	\section{Introduction}
	Let \(G=(V,E)\) be a nonempty finite simple graph. Its neighborhood
	complex is
	\[
	\NN(G)=
	\left\{
	A\subseteq V:
	\bigcap_{a\in A}N_G(a)\ne\varnothing
	\right\}.
	\]
	Lov\'asz introduced this complex in his proof of Kneser's conjecture
	\cite{Lovasz1978}. More recently, Fr\"oberg studied the Stanley--Reisner
	rings of neighborhood complexes, determined the \(2\)-linear resolutions
	for several graph families, and initiated a systematic investigation of
	the \(2\)-linear case \cite{Froberg2026}. We study the structure of the
	graphs \(G\) for which \(I_{\NN(G)}\) has a \(2\)-linear resolution and
	determine the corresponding homological invariants.

	The structural question is governed by the classical theory of hypertrees.
	Let
	\(\Ho(G)=(N_G(v))_{v\in V}\) be the indexed family of open neighborhoods,
	and let \(\CN(G)\) be the graph on \(V\) in which two vertices are adjacent
	when they have a common neighbor in \(G\). Identifying each index \(v\) with the
	indexed hyperedge \(N_G(v)\), we see that the line graph of \(\Ho(G)\) is
	\(\CN(G)\). Fr\"oberg's theorem and the classical criterion based on the Helly
	property and chordality therefore give
	\[
	I_{\NN(G)}\text{ has a \(2\)-linear resolution}
	\quad\Longleftrightarrow\quad
	\Ho(G)\text{ is a hypertree};
	\]
	see \cite{BrandstadtDraganChepoiVoloshin1998,Froberg1990}. This criterion
	has the following direct formulation in terms of \(G\): the
	ideal is \(2\)-linear exactly when \(G\) is bipartite, its indexed open
	neighborhoods are Helly, and every induced cycle of length at least eight has
	a filling from each color class. The proof identifies holes in
	the two half squares of a bipartite graph with induced cycles that do not
	admit a filling. It follows that the class properly contains the chordal bipartite 
	graphs without isolated vertices. For cactus graphs, the condition reduces to 
	the absence of isolated vertices and the requirement that every cycle have 
	length four.

	Under these equivalent structural conditions, set \(H=\CN(G)\). Then
	\(\NN(G)=\Cl(H)\) with \(H\) chordal. Hochster's
	formula determines every squarefree multigraded Betti number from the numbers
	of connected components of the induced subgraphs \(H[W]\). We also derive a
	closed formula for the total Betti numbers in terms of the numbers of vertex
	subsets having a common neighbor. In particular, if \(G\) has \(n\) vertices
	and \(c\) connected components, then
	\[
	\beta_{n-1,n}(\Bbbk[\NN(G)])=2c-1,
	\qquad
	\pd_S\Bbbk[\NN(G)]=n-1,
	\qquad
	\depth_S\Bbbk[\NN(G)]=1.
	\]
	These identities yield the initially Cohen--Macaulay and Cohen--Macaulay
	classifications and a multiplicity formula in terms of distinct maximum open
	neighborhoods.

	A second formula separates the contribution of the degree sequence from
	overlaps caused by repeated common neighbors. It gives explicit Betti
	formulas for
	bipartite graphs without \(K_{2,3}\), cactus graphs, pseudoforests, and
	forests.
	For square cactus graphs, the Betti table determines all degree
	multiplicities at least three, while the cycle count determines the remaining
	two. For forests, the Betti table determines the complete degree sequence. We
	also apply the closed neighborhood ideal to show that the dominance complex
	has a \(2\)-linear Stanley--Reisner ideal precisely for nontrivial stars.

	Section~\ref{sec:background} records the classical reduction and the basic
	conventions. Section~\ref{sec:structure} gives the reformulation in graph
	terms, and Section~\ref{sec:constructions} presents constructions and
	closure properties. Sections~\ref{sec:betti} determines the Betti numbers and
	the resulting properties of the Stanley--Reisner ring. Section~\ref{sec:overlap}
	introduces the overlap corrections, Section~\ref{sec:pseudoforests} treats
	cactus graphs, pseudoforest, and forests, and Section~\ref{sec:dominance} concerns dominance complexes.

\section{Preliminaries and the classical reduction}\label{sec:background}
	
	This section fixes the algebraic and combinatorial conventions and records the
	classical reduction to the indexed hypergraph of open neighborhoods.
	Throughout, \(G=(V,E)\) is a
	nonempty finite simple graph, \(\Bbbk\) is an arbitrary field, and
	\[
	S=\Bbbk[x_v:v\in V]
	\]
	is standard graded by \(\deg x_v=1\). Every simplicial complex is regarded as
	a complex on a fixed ground set, which may contain elements that belong to no
	face. We write
	\[
	\beta_{i,j}(M)=\dim_{\Bbbk}\operatorname{Tor}_i^S(M,\Bbbk)_j.
	\]
	For a finitely generated graded \(S\)-module \(M\), we use the conventions
	\[
	\reg(M)=\max\{j-i:\beta_{i,j}(M)\ne0\},
	\qquad
	\pd(M)=\max\{i:\beta_{i,j}(M)\ne0\text{ for some }j\}.
	\]
	The assertion that \(I_{\NN(G)}\) has a \(2\)-linear resolution means that
	it is generated in degree two and
	\[
	\beta_{i,j}(I_{\NN(G)})=0
	\qquad\text{for }j\ne i+2.
	\]
	We use \(\binom ab=0\) whenever \(b<0\) or \(b>a\).
	We write
	\[
	d(G)=\max_{v\in V}\deg_G(v)
	\]
	for the maximum degree of \(G\).
	
	\begin{definition}
		Let \(\Delta\) be a nonempty simplicial complex. Its \emph{initial
			dimension} is
		\[
		\operatorname{indim}\Delta
		=
		\min\{\dim F:F\text{ is a facet of }\Delta\}.
		\]
		For a nonzero finitely generated graded \(S\)-module \(M\), set
		\[
		\operatorname{indim} M
		=
		\min\{\dim(S/\mathfrak p):\mathfrak p\in\operatorname{Ass}(M)\}.
		\]
		Since a Stanley--Reisner ideal is radical, the associated primes of
		\(\Bbbk[\Delta]\) are its minimal primes, which correspond to the facets
		of \(\Delta\). Hence
		\[
		\operatorname{indim}\Bbbk[\Delta]
		=
		\operatorname{indim}\Delta+1.
		\]
		The complex \(\Delta\) is \emph{initially Cohen--Macaulay over
			\(\Bbbk\)} if
		\[
		\depth\Bbbk[\Delta]
		=
		\operatorname{indim}\Bbbk[\Delta],
		\]
		or, equivalently,
		\(\depth\Bbbk[\Delta]=\operatorname{indim}\Delta+1\);
		see \cite{NamiqInitiallyCM}.
	\end{definition}
	
	\begin{definition}
		The \emph{common neighbor graph} \(\CN(G)\) is the graph on \(V\) in which
		distinct vertices \(u,v\) are adjacent if and only if
		\[
		N_G(u)\cap N_G(v)\ne\varnothing.
		\]
		The open neighborhoods of \(G\) are \emph{Helly} if every pairwise
		intersecting indexed subfamily of \((N_G(v))_{v\in V}\) has nonempty total
		intersection. We use the standard Helly convention that singleton
		subfamilies are included; the terminology follows Groshaus and Szwarcfiter
		\cite{GroshausSzwarcfiter2008}.
	\end{definition}
	
	The singleton convention implies that a graph whose open neighborhoods form a
	Helly family has no isolated vertices. We write \(\Cl(H)\) for the clique
	complex of a graph \(H\).
	
	\begin{lemma}\label{lem:helly-flag}
		For a nonempty finite simple graph \(G\), the following are equivalent:
		\begin{enumerate}[label=\textup{(\roman*)}]
			\item the open neighborhoods of \(G\) form a Helly family;
			\item every nonempty clique \(Q\) of \(\CN(G)\) is contained in
			\(N_G(w)\) for some \(w\in V\);
			\item \(\NN(G)=\Cl(\CN(G))\);
			\item \(\NN(G)\) is flag on the ground set \(V\).
		\end{enumerate}
	\end{lemma}
	
	\begin{proof}
		Let \(Q\subseteq V\) be nonempty. The indexed family
		\((N_G(q))_{q\in Q}\) is pairwise intersecting exactly when every two
		vertices of \(Q\) have a common neighbor, or equivalently, when \(Q\) is a
		clique of \(\CN(G)\). Moreover,
		\[
		\bigcap_{q\in Q}N_G(q)\ne\varnothing
		\quad\Longleftrightarrow\quad
		Q\subseteq N_G(w)\text{ for some }w\in V.
		\]
		Thus \textup{(i)} and \textup{(ii)} are equivalent.
		
		Every face of \(\NN(G)\) is contained in an open neighborhood. Any two
		vertices of such a face therefore have a common neighbor, and hence
		\[
		\NN(G)\subseteq\Cl(\CN(G)).
		\]
		Condition \textup{(ii)} gives the reverse inclusion, proving the equivalence
		of \textup{(ii)} and \textup{(iii)}.
		
		Finally, the edges of \(\CN(G)\) are precisely the faces of \(\NN(G)\)
		with two vertices. Therefore \(\NN(G)\) is flag on the prescribed ground set
		\(V\) if and only if it equals \(\Cl(\CN(G))\). This formulation also covers
		isolated vertices. If \(v\) is isolated in \(G\), then \(\{v\}\) is a
		clique of \(\CN(G)\) but not a face of \(\NN(G)\); consequently, both
		\textup{(iii)} and \textup{(iv)} fail.
	\end{proof}
	
	A hypergraph on \(V\) is a \emph{hypertree} if every hyperedge is nonempty and
	there exists a tree on \(V\) in which each hyperedge induces a subtree. We
	treat repeated open neighborhoods as indexed hyperedges. Passing to the family
	of distinct hyperedges preserves both the Helly property and subtree
	representability. Conversely, duplicating a nonempty hyperedge adds a true twin
	to the line graph and therefore preserves chordality. The classical hypertree
	criterion consequently applies to the indexed family without modification.
	
	\begin{proposition}\label{prop:background}
		For a nonempty finite simple graph \(G\), the following are equivalent:
		\begin{enumerate}[label=\textup{(\roman*)}]
			\item \(I_{\NN(G)}\) has a \(2\)-linear resolution;
			\item the open neighborhoods of \(G\) form a Helly family, and \(\CN(G)\) is chordal;
			\item \(\Ho(G)=(N_G(v))_{v\in V}\) is a hypertree.
		\end{enumerate}
		The conditions are independent of \(\operatorname{char}\Bbbk\), and any
		graph satisfying them has no isolated vertices.
	\end{proposition}
	
	\begin{proof}
		The ideal \(I_{\NN(G)}\) is nonzero because no open neighborhood contains the
		full vertex set. Assume first that \textup{(i)} holds. Since a \(2\)-linear
		resolution forces \(I_{\NN(G)}\) to be generated in degree \(2\), the complex
		\(\NN(G)\) is flag on \(V\). Lemma~\ref{lem:helly-flag} then shows that the open neighborhoods of \(G\)
		form a Helly family and yields the equalities
		\[
		\NN(G)=\Cl(\CN(G)),
		\qquad
		I_{\NN(G)}=I\bigl(\overline{\CN(G)}\bigr).
		\]
		Fr\"oberg's theorem \cite[Theorem~1]{Froberg1990} now implies that
		\(\CN(G)\) is chordal. Hence \textup{(i)} implies \textup{(ii)}.
		
		Conversely, assume \textup{(ii)}. Lemma~\ref{lem:helly-flag} again yields
		\[
		I_{\NN(G)}=I\bigl(\overline{\CN(G)}\bigr).
		\]
		The right side is the edge ideal of the complement of the chordal graph
		\(\CN(G)\). Fr\"oberg's theorem therefore gives a \(2\)-linear resolution,
		proving the equivalence of \textup{(i)} and \textup{(ii)}.
		
		Under the identification of \(v\) with the indexed hyperedge \(N_G(v)\), the
		line graph of \(\Ho(G)\) is \(\CN(G)\). The classical characterization of hypertrees by the Helly property
		and chordality \cite{BrandstadtDraganChepoiVoloshin1998} proves the
		equivalence of \textup{(ii)} and \textup{(iii)}. All three conditions are
		combinatorial and are therefore independent of \(\operatorname{char}\Bbbk\).
		Finally, \textup{(ii)} excludes isolated vertices: for every \(v\in V\), the
		singleton indexed family \(\{N_G(v)\}\) must have nonempty intersection.
	\end{proof}

	Our convention for the ground set makes isolated vertices algebraically visible.
	The following observation records exactly how they alter the resolution.

	\begin{proposition}\label{prop:isolated-vertices}
		Let \(Z\) be the set of isolated vertices of \(G\), let \(t=|Z|\), and put
		\(G^{\circ}=G-Z\). Assume first that \(G^{\circ}\) is nonempty. With
		\[
		S^{\circ}=\Bbbk[x_v:v\in V(G^{\circ})],
		\qquad S=S^{\circ}[x_z:z\in Z],
		\]
		one has
		\[
		I_{\NN(G)}=I_{\NN(G^{\circ})}S+(x_z:z\in Z)
		\]
		and
		\begin{equation}\label{eq:isolated-convolution}
		\beta^{S}_{i,j}\bigl(\Bbbk[\NN(G)]\bigr)
		=
		\sum_{a=0}^{\min\{t,i\}}
		\binom ta\,
		\beta^{S^{\circ}}_{i-a,j-a}
		\bigl(\Bbbk[\NN(G^{\circ})]\bigr).
		\end{equation}
		If every vertex of \(G\) is isolated, then
		\(I_{\NN(G)}=(x_v:v\in V(G))\) and
		\(\beta^S_{i,i}=\binom{|V(G)|}{i}\).
	\end{proposition}

	\begin{proof}
		No face of \(\NN(G)\) contains a vertex of \(Z\), while a subset of
		\(V(G^{\circ})\) is a face of \(\NN(G)\) exactly when it is a face of
		\(\NN(G^{\circ})\). This proves the ideal identity. Hence
		\[
		\Bbbk[\NN(G)]
		\cong
		\Bbbk[\NN(G^{\circ})]
		\otimes_{\Bbbk}
		\Bbbk[x_z:z\in Z]/(x_z:z\in Z).
		\]
		Tensoring a minimal \(S^{\circ}\)-resolution of the first factor with the
		Koszul resolution of the second factor gives a minimal \(S\)-resolution and
		yields \eqref{eq:isolated-convolution}. When all vertices are isolated, the resolution is the Koszul resolution of
		the residue field.
	\end{proof}
	
	\section{A structural reformulation in graph terms}\label{sec:structure}
	
	Proposition~\ref{prop:background} can be stated directly in terms of the
	cycles and neighborhoods of \(G\).
	
	Let \(G=(X\sqcup Y,E)\) be bipartite. Vertices in different color classes
	cannot have a common neighbor, and hence
	\begin{equation}\label{eq:half-square-decomposition}
		\CN(G)=\CN(G)[X]\sqcup\CN(G)[Y].
	\end{equation}
	The two summands are the half squares of \(G\) on \(X\) and on \(Y\)
	\cite{LeLe2019}.

	\begin{definition}
		Let
		\[
		C=x_1 y_1 x_2 y_2\cdots x_q y_q x_1,\qquad q\ge4,
		\]
		be an induced cycle of \(G=(X\sqcup Y,E)\), with \(x_i\in X\) and
		\(y_i\in Y\). A \emph{filling on \(X\)} of \(C\) is a common neighbor in
		\(G\) of two nonconsecutive vertices among \(x_1,\ldots,x_q\). A
		\emph{filling on \(Y\)} is defined symmetrically. The cycle is
		\emph{filled on both sides} if it has a filling on \(X\) and a filling on \(Y\).
	\end{definition}
	
	Because \(C\) is induced, every vertex that provides a filling lies outside
	\(V(C)\). A filling on \(X\) is provided by a vertex of \(Y\), whereas a
	filling on \(Y\) is provided by a vertex of \(X\). Thus the label records the
	color class of the cycle vertices that share the additional common neighbor.
	
	\begin{example}\label{ex:bifilled-cycle}
		Let $C=x_1y_1x_2y_2x_3y_3x_4y_4x_5y_5x_1$ be an induced \(10\)-cycle. If \(z_Y\in Y\) and \(z_X\in X\) satisfy
		\[
		N_G(z_Y)\cap V(C)=\{x_1,x_3,x_5\},
		\qquad
		N_G(z_X)\cap V(C)=\{y_1,y_3\},
		\]
		then \(z_Y\) provides a filling on \(X\), while \(z_X\) provides a filling on \(Y\).
		Hence \(C\) is filled on both sides. Figure~\ref{fig:bifilled-cycle} illustrates the two fillings.
		
		\begin{figure}[H]
			\centering
\begin{tikzpicture}[
				scale=0.7,
				xvertex/.style={circle,draw,minimum size=6.5mm,inner sep=1pt},
				yvertex/.style={rectangle,rounded corners=1.5pt,draw,
					minimum size=6.5mm,inner sep=1pt},
				cycleedge/.style={thick},
				fillingedge/.style={thick,densely dashed}
				]
				
				\begin{scope}[xshift=-3.7cm]
					\node[xvertex] (ax1) at (90:2.15) {$x_1$};
					\node[yvertex] (ay1) at (54:2.15) {$y_1$};
					\node[xvertex] (ax2) at (18:2.15) {$x_2$};
					\node[yvertex] (ay2) at (-18:2.15) {$y_2$};
					\node[xvertex] (ax3) at (-54:2.15) {$x_3$};
					\node[yvertex] (ay3) at (-90:2.15) {$y_3$};
					\node[xvertex] (ax4) at (-126:2.15) {$x_4$};
					\node[yvertex] (ay4) at (-162:2.15) {$y_4$};
					\node[xvertex] (ax5) at (162:2.15) {$x_5$};
					\node[yvertex] (ay5) at (126:2.15) {$y_5$};
					\node[yvertex] (azY) at (0,0) {$z_Y$};
					\draw[cycleedge]
					(ax1)--(ay1)--(ax2)--(ay2)--(ax3)--(ay3)--(ax4)--(ay4)--(ax5)--(ay5)--cycle;
					\draw[fillingedge] (azY)--(ax1);
					\draw[fillingedge] (azY)--(ax3);
					\draw[fillingedge] (azY)--(ax5);
					\node at (0,-3) {\textup{(a)} filling on \(X\) by \(z_Y\)};
				\end{scope}
				
				\begin{scope}[xshift=3.7cm]
					\node[xvertex] (bx1) at (90:2.15) {$x_1$};
					\node[yvertex] (by1) at (54:2.15) {$y_1$};
					\node[xvertex] (bx2) at (18:2.15) {$x_2$};
					\node[yvertex] (by2) at (-18:2.15) {$y_2$};
					\node[xvertex] (bx3) at (-54:2.15) {$x_3$};
					\node[yvertex] (by3) at (-90:2.15) {$y_3$};
					\node[xvertex] (bx4) at (-126:2.15) {$x_4$};
					\node[yvertex] (by4) at (-162:2.15) {$y_4$};
					\node[xvertex] (bx5) at (162:2.15) {$x_5$};
					\node[yvertex] (by5) at (126:2.15) {$y_5$};
					\node[xvertex] (bzX) at (0,0) {$z_X$};
					\draw[cycleedge]
					(bx1)--(by1)--(bx2)--(by2)--(bx3)--(by3)--(bx4)--(by4)--(bx5)--(by5)--cycle;
					\draw[fillingedge] (bzX)--(by1);
					\draw[fillingedge] (bzX)--(by3);
					\node at (0,-3){\textup{(b)} filling on \(Y\) by \(z_X\)};
				\end{scope}
			\end{tikzpicture}
			\caption{Solid edges form the induced cycle, and dashed edges show the two fillings}
			\label{fig:bifilled-cycle}
		\end{figure}
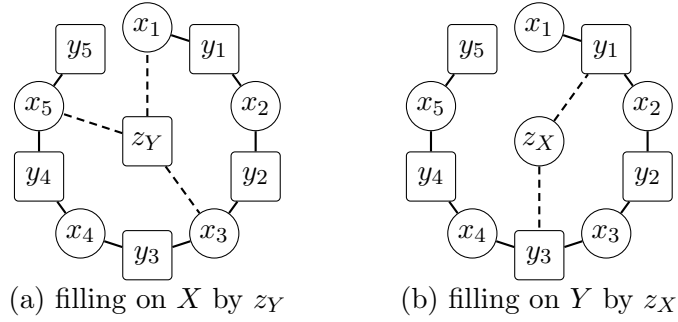
	\end{example}
	
	\begin{lemma}\label{lem:cycle-correspondence}
		Let \(G=(X\sqcup Y,E)\) be bipartite and \(q\ge4\).
		\begin{enumerate}[label=\textup{(\alph*)}]
			\item \(\CN(G)[X]\) contains an induced \(C_q\) if and only if \(G\)
			contains an induced \(C_{2q}\) with no filling on the \(X\) side.
			\item \(\CN(G)[Y]\) contains an induced \(C_q\) if and only if \(G\)
			contains an induced \(C_{2q}\) with no filling on the \(Y\) side.
		\end{enumerate}
	\end{lemma}
	
	\begin{proof}
		It is enough to prove \textup{(a)}, since interchanging \(X\) and \(Y\)
		gives \textup{(b)}.
		
		Suppose first that \(x_1,\ldots,x_q\) induce a cycle in \(\CN(G)[X]\),
		with indices read modulo \(q\). For each \(i\), select
		\[
		y_i\in N_G(x_i)\cap N_G(x_{i+1}).
		\]
		The selected vertices \(y_1,\ldots,y_q\) are distinct. If
		\(y_i=y_j\) for \(i\ne j\), then this vertex is adjacent to the endpoints of
		two different edges of the \(x\)-cycle. It would therefore create an edge of
		the common neighbor graph between two nonconsecutive vertices of the induced
		cycle. The same argument shows that, among the selected \(y\)-vertices,
		\(x_i\) is adjacent only to \(y_{i-1}\) and \(y_i\). Consequently,
		\[
		x_1y_1x_2y_2\cdots x_qy_qx_1
		\]
		is an induced \(C_{2q}\) in \(G\). This cycle has no filling on the \(X\)
		side, since a common neighbor of two nonconsecutive \(x_i\)'s would produce a
		chord of the
		induced cycle in \(\CN(G)[X]\).
		
		Conversely, suppose that
		\[
		x_1y_1x_2y_2\cdots x_qy_qx_1
		\]
		is an induced \(C_{2q}\) with no filling on the \(X\) side. Each consecutive pair
		\(x_i,x_{i+1}\) has the common neighbor \(y_i\), whereas no nonconsecutive
		pair of vertices in \(X\) on the cycle has a common neighbor. Hence
		\(x_1,\ldots,x_q\) induce a \(C_q\) in \(\CN(G)[X]\).
	\end{proof}
	
	\begin{proposition}\label{prop:cycle-filling}
		For a bipartite graph \(G=(X\sqcup Y,E)\), the graph \(\CN(G)\) is chordal
		if and only if every induced cycle of \(G\) of length at least eight is
		filled on both sides.
	\end{proposition}
	
	\begin{proof}
		By \eqref{eq:half-square-decomposition}, \(\CN(G)\) is chordal exactly when
		both half squares, \(\CN(G)[X]\) and \(\CN(G)[Y]\), are chordal.
		Lemma~\ref{lem:cycle-correspondence} shows that a hole \(C_q\), \(q\ge4\),
		in the half square on \(X\) corresponds to an induced \(C_{2q}\) in \(G\) with
		no filling on the \(X\) side. The symmetric statement holds for the half
		square on \(Y\). Thus neither half square contains a hole if and only if every
		induced cycle of
		\(G\) of length at least eight has a filling on \(X\) and a filling on
		\(Y\).
	\end{proof}
	
	Proposition~\ref{prop:cycle-filling} expresses the \(X\)-chordal and
	\(Y\)-chordal conditions in terms of half squares; see
	\cite[Theorems~6--7]{BrandstadtDraganChepoiVoloshin1998}. In that paper,
	the side label refers to the color class containing the bridge
	vertex. Here,
	a filling is labeled instead by the color class of the two cycle vertices that
	share that bridge vertex as a common neighbor. The following theorem therefore
	combines that classical bipartite hypertree characterization with
	Proposition~\ref{prop:background}.
	
	\begin{theorem}\label{thm:structural}
		For a nonempty finite simple graph \(G\), the following are equivalent:
		\begin{enumerate}[label=\textup{(\roman*)}]
			\item \(I_{\NN(G)}\) has a \(2\)-linear resolution over \(\Bbbk\);
			\item \(G\) is bipartite, its open neighborhoods form a Helly family, and
			every induced cycle of length at least eight is filled on both sides.
		\end{enumerate}
	\end{theorem}
	
	\begin{proof}
		Assume first that \textup{(i)} holds. By
		Proposition~\ref{prop:background}, the open neighborhoods of \(G\) form a
		Helly family and \(\CN(G)\) is chordal. We show that \(G\) is bipartite.
		
		First, \(G\) is triangle free. If \(G\) contained a triangle, it could be
		extended to a maximal clique \(Q\). Since \(|Q|\ge3\), every two vertices
		of \(Q\) have a third vertex of \(Q\) as a common neighbor. Hence \(Q\) is a
		clique of \(\CN(G)\). Lemma~\ref{lem:helly-flag} implies that \(Q\) is a face
		of \(\NN(G)\), so there exists \(w\in V(G)\) such that
		\[
		Q\subseteq N_G(w).
		\]
		Because \(G\) has no loops, \(w\notin Q\). Therefore \(Q\cup\{w\}\) is a
		clique of \(G\), contradicting the maximality of \(Q\). Hence \(G\) is triangle free.
		
		Suppose, toward a contradiction, that \(G\) is not bipartite. Let
		\[
		C=v_0v_1\cdots v_{m-1}v_0
		\]
		be a shortest odd cycle.
		The absence of triangles gives \(m\ge5\), and the minimality of \(m\)
		implies that \(C\) is induced. Since \(m\) is odd, multiplication by \(2\)
		permutes the residue classes modulo \(m\). The cyclic order
		\[
		v_0,v_2,v_4,\ldots,v_{2(m-1)},v_0,
		\]
		with indices taken modulo \(m\), therefore visits every vertex of \(C\).
		Consecutive vertices in this order are at distance two on \(C\), so they
		have a common neighbor on \(C\). They consequently form an \(m\)-cycle in
		\(\CN(G)\).
		
		We claim that this cycle is induced. Suppose that two vertices \(v_i\) and
		\(v_j\), nonconsecutive in this order, have a common neighbor
		\(w\). If \(w\in V(C)\), then the inducedness of \(C\) forces \(v_i\) and
		\(v_j\) to be the two neighbors of \(w\) on \(C\). They would then be
		consecutive in this order, a contradiction. Hence
		\(w\notin V(C)\).
		
		Exactly one of the two \(v_i\)--\(v_j\) arcs of \(C\) has odd length. Let
		\(\ell\) denote its length. If \(\ell=1\), then \(v_i,v_j,w\) form a triangle.
		If \(\ell=m-2\), then the complementary arc has length two, so \(v_i\) and
		\(v_j\) are consecutive in this order. Both cases are impossible,
		and therefore
		\[
		3\le\ell\le m-4.
		\]
		The odd arc, together with the edges \(v_iw\) and \(wv_j\), forms an odd
		cycle of length \(\ell+2<m\). This contradicts the choice of \(C\). Thus the
		cycle obtained from this order is induced and has length \(m\ge5\) in
		\(\CN(G)\),
		contrary to chordality. Hence \(G\) is bipartite.
		
		Proposition~\ref{prop:cycle-filling} now implies that every induced cycle of
		\(G\) of length at least eight is filled on both sides. Together with the
		Helly property of the open neighborhoods, this proves
		\textup{(ii)}.
		
		Conversely, assume \textup{(ii)}. By Proposition~\ref{prop:cycle-filling},
		the condition on induced cycles implies that \(\CN(G)\) is chordal. Since
		the open neighborhoods of \(G\) also form a Helly family,
		Proposition~\ref{prop:background} shows that \(I_{\NN(G)}\) has a \(2\)-linear resolution. Therefore
		\textup{(i)} holds.
	\end{proof}
	
	\begin{remark}
		A bipartite graph is \emph{chordal bipartite} if it has no induced cycle of
		length at least six. An induced \(C_6\) gives a triangle, rather than a
		hole, in each half square. The three vertices in either color class also
		determine pairwise intersecting open neighborhoods. If the open
		neighborhoods of \(G\) form a Helly family, each triple therefore
		has a common neighbor outside
		the cycle. Thus chordality of the half squares alone does not exclude an induced
		\(C_6\); the Helly condition in Theorem~\ref{thm:structural} is essential.
	\end{remark}
	
	\begin{corollary}\label{cor:chordal-bipartite}
		If \(G\) is a chordal bipartite graph without isolated vertices, then
		\(I_{\NN(G)}\) has a \(2\)-linear resolution.
	\end{corollary}
	
	\begin{proof}
		A chordal bipartite graph has no induced cycle of length at least six, so the
		condition on induced cycles in Theorem~\ref{thm:structural} is vacuous. It
		remains to prove that the open neighborhoods form a Helly family.
		
		Let \((N_G(x))_{x\in Q}\) be a pairwise intersecting indexed subfamily with at
		least two members. Since \(G\) is bipartite, all vertices of \(Q\) lie in the
		same color class. After interchanging the color classes if necessary, assume
		that \(Q\subseteq X\). If the total intersection were empty, there would be a subfamily indexed by
		\(Q'\subseteq Q\) that is minimal among those with empty intersection. Then
		\(|Q'|\ge3\), so there exist distinct \(x_1,x_2,x_3\in Q'\). By minimality, for each \(i\in\{1,2,3\}\) there
		exists
		\[
		y_i\in
		\bigcap_{x\in Q'\setminus\{x_i\}}N_G(x)
		\setminus N_G(x_i).
		\]
		The vertices \(y_1,y_2,y_3\) are distinct. Among the six selected vertices,
		\(y_i\) is adjacent to the two vertices \(x_j\) with \(j\ne i\), but not to
		\(x_i\). Hence
		\[
		x_1y_2x_3y_1x_2y_3x_1
		\]
		is an induced \(C_6\), contradicting chordal bipartiteness. Every pairwise
		intersecting subfamily with at least two members therefore has nonempty total
		intersection. Singleton subfamilies also intersect because \(G\) has no
		isolated vertices. Thus the open neighborhoods of \(G\) form a Helly family, and
		Theorem~\ref{thm:structural} gives the conclusion.
	\end{proof}
	
	\begin{remark}
		Corollary~\ref{cor:chordal-bipartite} applies, in particular, to every
		forest without isolated vertices and every chain graph without isolated
		vertices. It also recovers the \(2\)-linearity assertions for paths,
		stars, complete bipartite graphs, and \(2\times n\) grids obtained in
		\cite[Theorems~3.2, 3.4, 3.8, and 3.12]{Froberg2026}.
	\end{remark}

	A graph is a \emph{cactus graph} if any two cycles have at most one vertex in
	common. The structural criterion gives a particularly simple characterization
	within this class.

	\begin{theorem}\label{thm:cactus-characterization}
		Let \(G\) be a nonempty finite simple cactus graph. Then
		\(I_{\NN(G)}\) has a \(2\)-linear resolution if and only if \(G\) has no
		isolated vertices and every cycle of \(G\) has length four.
	\end{theorem}

	\begin{proof}
		Suppose first that \(I_{\NN(G)}\) has a \(2\)-linear resolution.
		Proposition~\ref{prop:background} shows that \(G\) has no isolated vertices,
		and Theorem~\ref{thm:structural} shows that \(G\) is bipartite and that its
		open neighborhoods form a Helly family.

		Every cycle of a cactus graph is induced, since a chord would produce two
		cycles with more than one common vertex. Moreover, no vertex outside a cycle
		can be adjacent to two vertices of that cycle, since the two additional edges
		and either suitable arc of the cycle would produce another cycle meeting the
		original one in at least two vertices.

		Because \(G\) is bipartite, every cycle has even length. Let \(C\) be a cycle
		of length at least eight. Theorem~\ref{thm:structural} requires a filling on
		each side of \(C\). Such a filling lies outside the induced cycle and is
		adjacent to two vertices of \(C\), contradicting the cactus property.

		A cycle of length six is also impossible. Let
		\[
		C=x_1y_1x_2y_2x_3y_3x_1.
		\]
		The neighborhoods \(N_G(x_1),N_G(x_2),N_G(x_3)\) intersect pairwise, since
		\[
		y_1\in N_G(x_1)\cap N_G(x_2),\qquad
		y_2\in N_G(x_2)\cap N_G(x_3),\qquad
		y_3\in N_G(x_3)\cap N_G(x_1).
		\]
		The Helly property gives a common neighbor of \(x_1,x_2,x_3\). This vertex
		lies outside \(C\) and is adjacent to at least two vertices of \(C\), again
		contradicting the cactus property. Hence every cycle of \(G\) has length four.

		Conversely, suppose that \(G\) has no isolated vertices and every cycle has
		length four. Then \(G\) is bipartite and has no induced cycle of length at
		least six. Thus \(G\) is chordal bipartite, and
		Corollary~\ref{cor:chordal-bipartite} gives the desired \(2\)-linear
		resolution.
	\end{proof}
	\begin{corollary}\label{cor:unique_cycle}
		The cycle graph \(C_4\) is the unique cycle graph \(C_n\) for which
		\(I_{\NN(C_n)}\) has a \(2\)-linear resolution.
	\end{corollary}
	
	\begin{proof}
		This is immediate from Theorem~\ref{thm:cactus-characterization}.
	\end{proof}
	
	\section{Constructions and closure properties}\label{sec:constructions}
	
	The first result gives a construction, followed by three closure properties.
	
	\begin{proposition}\label{prop:two-sided-cone}
		Let \(H=(X\sqcup Y,E(H))\) be a finite bipartite graph. Let \(x_0\) and
		\(y_0\) be
		two new vertices, and put
		\[
		\widehat X=X\cup\{x_0\},
		\qquad
		\widehat Y=Y\cup\{y_0\}.
		\]
		Define the bipartite graph \(\widehat H\) with color classes
		\(\widehat X\) and \(\widehat Y\) by
		\[
		E(\widehat H)
		=
		E(H)
		\cup
		\{x_0y:y\in Y\cup\{y_0\}\}
		\cup
		\{xy_0:x\in X\}.
		\]
		Then \(I_{\NN(\widehat H)}\) has a \(2\)-linear resolution over
		\(\Bbbk\).
	\end{proposition}
	
	\begin{proof}
		Every two vertices of \(\widehat X\) have the common neighbor \(y_0\), and
		every two vertices of \(\widehat Y\) have the common neighbor \(x_0\).
		Vertices in different color classes cannot have a common neighbor in a
		bipartite graph. Hence
		\[
		\CN(\widehat H)=K_{\widehat X}\sqcup K_{\widehat Y},
		\]
		which is chordal.
		
		It remains to prove that the open neighborhoods form a Helly family. Let
		\(Q\) be a
		clique of \(\CN(\widehat H)\). The displayed decomposition shows that either
		\(Q\subseteq\widehat X\) or \(Q\subseteq\widehat Y\). In the first case,
		\[
		Q\subseteq N_{\widehat H}(y_0),
		\]
		while in the second,
		\[
		Q\subseteq N_{\widehat H}(x_0).
		\]
		Thus every clique of \(\CN(\widehat H)\) is contained in an open
		neighborhood. Lemma~\ref{lem:helly-flag} implies that the open neighborhoods of
		\(\widehat H\) form a Helly family, and Proposition~\ref{prop:background} gives the
		result.
	\end{proof}
	
	For \(H=C_{2q}\) with \(q\ge3\), this class properly
	contains the chordal bipartite graphs without isolated vertices: the original
	cycle remains induced in \(\widehat H\), while the two new vertices provide
	fillings from the two color classes. Thus the construction produces admissible
	graphs with induced cycles of arbitrary even length.
	
	A \emph{false twin} of a vertex \(v\) is a nonadjacent vertex with the same open neighborhood as \(v\).
	
	\begin{proposition}
		\label{prop:closure-operations}
		The class of finite simple graphs \(G\) for which
		\(I_{\NN(G)}\) has a \(2\)-linear resolution is closed under the
		following operations:
		\begin{enumerate}[label=\textup{(\alph*)}]
			\item taking a nonempty finite disjoint union;
			\item adding a pendant vertex adjacent to an arbitrary vertex;
			\item adding a false twin of an arbitrary vertex.
		\end{enumerate}
	\end{proposition}
	
	\begin{proof}
		The proof uses Proposition~\ref{prop:background} and the clique formulation
		in Lemma~\ref{lem:helly-flag}.

		For \textup{(a)}, if \(G=G_1\sqcup\cdots\sqcup G_t\), then
		\[
		\CN(G)=\CN(G_1)\sqcup\cdots\sqcup\CN(G_t).
		\]
		This graph is chordal, and every clique lies in one component, where it is
		contained in an open neighborhood. Hence the open neighborhoods of \(G\)
		are Helly.

		For \textup{(b)}, let \(p\) be pendant at \(v\). The old induced subgraph
		of \(\CN(G')\) is \(\CN(G)\), and
		\[
		N_{\CN(G')}(p)=N_G(v),
		\]
		which is a clique because its vertices have the common neighbor \(v\).
		Thus \(p\) is simplicial and \(\CN(G')\) is chordal. A clique not
		containing \(p\) has an old common neighbor, while a clique containing
		\(p\) is contained in \(N_{G'}(v)\). Hence the new neighborhood family is
		Helly.

		For \textup{(c)}, let \(v'\) be a false twin of \(v\). Then
		\(\CN(G')[V(G)]=\CN(G)\), and \(v,v'\) are true twins in
		\(\CN(G')\). Adding a true twin to a chordal graph preserves chordality:
		an induced cycle containing only the new twin can be transferred to the old
		one, whereas a cycle containing both has a chord through their equal closed
		neighborhoods. Finally, if a clique \(Q\) contains \(v'\), replace \(v'\)
		by \(v\); an old common neighbor of the resulting clique is also a common
		neighbor of \(Q\), since \(N_{G'}(v')=N_G(v)\). Thus the open
		neighborhoods remain Helly in all three cases.
	\end{proof}
	
	The false twin hypothesis is essential: adding a true twin to an endpoint of
	\(K_2\) produces \(K_3\), which is excluded by
	Theorem~\ref{thm:structural}.
	
	\section{Multigraded and graded Betti numbers}\label{sec:betti}
	
	Throughout this section, assume that \(I_{\NN(G)}\) has a \(2\)-linear
	resolution, and set \(n=|V|\), \(H=\CN(G)\), and
	\(R=\Bbbk[\NN(G)]\). The common neighbor graph gives the multigraded
	Betti numbers, while subsets of the original graph with a common neighbor give
	the total Betti numbers.
	Proposition~\ref{prop:background} gives
	\[
	\NN(G)=\Cl(H)\qquad\text{and}\qquad H\text{ is chordal}.
	\]
	For \(W\subseteq V\), let \(c_H(W)\) denote the number of connected
	components of \(H[W]\), and write
	\[
	\beta_{i,W}(R)
	=
	\dim_{\Bbbk}\operatorname{Tor}_i(R,\Bbbk)_W
	\]
	for the squarefree multigraded Betti number indexed by \(W\).

	\begin{lemma}\label{lem:chordal-clique-contractible}
		If \(H\) is a nonempty connected chordal graph, then its clique complex
		\(\Cl(H)\) is contractible.
	\end{lemma}
	
	\begin{proof}
		We argue by induction on \(|V(H)|\). The result is immediate when \(H\) has
		one vertex. Assume that \(|V(H)|>1\). Since \(H\) is chordal, it has a simplicial vertex \(v\). Connectedness gives
		\(N_H(v)\ne\varnothing\).
		
		The graph \(H-v\) remains connected. If a path in \(H\) contains a
		segment \(x,v,y\), then \(x\) and \(y\) are adjacent because \(N_H(v)\) is a
		clique. Replacing the segment \(x,v,y\) by the edge \(xy\) removes \(v\)
		from the path. The graph \(H-v\) is also chordal, so the induction hypothesis
		implies that \(\Cl(H-v)\) is contractible.
		
		Every face of \(\Cl(H)\) containing \(v\) lies in the simplex on
		\(N_H[v]\). This simplex meets \(\Cl(H-v)\) in the nonempty simplex on
		\(N_H(v)\). Collapsing the simplex on \(N_H[v]\) onto this common face,
		relative to the intersection, collapses \(\Cl(H)\) onto \(\Cl(H-v)\).
		Since \(\Cl(H-v)\) is contractible, so is \(\Cl(H)\).
	\end{proof}

	\begin{theorem}\label{thm:all-betti}
		The multigraded Betti numbers of \(R=\Bbbk[\NN(G)]\) are
		\[
		\beta_{0,\varnothing}(R)=1,
		\qquad
		\beta_{0,W}(R)=0\quad(W\ne\varnothing),
		\]
		and, for \(i\ge1\),
		\begin{equation}\label{eq:multigraded}
			\beta_{i,W}(R)=
			\begin{cases}
				c_H(W)-1,& |W|=i+1,\\
				0,& |W|\ne i+1.
			\end{cases}
		\end{equation}
		Consequently, the only total Betti numbers in positive homological degrees are
		\begin{equation}\label{eq:component-total}
			\beta_{i,i+1}(R)=
			\sum_{\substack{W\subseteq V\\|W|=i+1}}
			\bigl(c_H(W)-1\bigr),\qquad 1\le i\le n-1.
		\end{equation}
		Equivalently, all Betti numbers of the ideal are
		\[
		\beta_{p,p+2}(I_{\NN(G)})
		=
		\sum_{\substack{W\subseteq V\\|W|=p+2}}
		\bigl(c_H(W)-1\bigr),\qquad 0\le p\le n-2,
		\]
		and \(\beta_{p,j}(I_{\NN(G)})=0\) for \(j\ne p+2\).
	\end{theorem}
	
	\begin{proof}
		Let \(W\subseteq V\). Since \(H\) is chordal, the induced subgraph \(H[W]\)
		is chordal as well. By Lemma~\ref{lem:chordal-clique-contractible}, the clique
		complex of each connected component of \(H[W]\) is contractible. Therefore,
		for \(W\ne\varnothing\), the complex \(\Cl(H[W])\) is a disjoint union of
		\(c_H(W)\) contractible complexes, and its only nonzero reduced homology is
		\[
		\dim_{\Bbbk}\widetilde H_0(\Cl(H[W]);\Bbbk)=c_H(W)-1.
		\]
		
		Hochster's formula \cite{Hochster1977} now gives, for \(i\ge1\),
		\[
		\beta_{i,W}(R)
		=
		\dim_{\Bbbk}\widetilde H_{|W|-i-1}
		\bigl(\Cl(H[W]);\Bbbk\bigr)
		=
		\begin{cases}
			c_H(W)-1,&|W|=i+1,\\
			0,&|W|\ne i+1.
		\end{cases}
		\]
		For \(W=\varnothing\), the standard convention gives
		\(\beta_{0,\varnothing}(R)=1\), while all other multigraded Betti numbers in
		homological degree zero vanish. This agrees with the formula of Engstr\"om and
		Stamps \cite{EngstromStamps2013} for clique complexes of chordal
		graphs.
		
		Summing over all subsets \(W\) with \(|W|=i+1\) gives
		\eqref{eq:component-total}. Finally, the exact sequence
		\[
		0\longrightarrow I_{\NN(G)}\longrightarrow S\longrightarrow R
		\longrightarrow0
		\]
		shifts the positive homological degrees by one. Hence
		\[
		\beta_{p,p+2}(I_{\NN(G)})=\beta_{p+1,p+2}(R),
		\]
		which yields the asserted Betti numbers of the ideal.
	\end{proof}
	
	Formula~\eqref{eq:component-total} involves all induced subgraphs of \(H\).
	The next form uses only counts of common neighborhoods in the original graph.
	
	\begin{definition}
		For \(0\le r\le n\), let
		\[
		q_r(G)=
		\left|
		\left\{
		A\subseteq V:
		|A|=r,\ 
		\bigcap_{a\in A}N_G(a)\ne\varnothing
		\right\}
		\right|.
		\]
		Thus \(q_r(G)\) is the number of faces of \(\NN(G)\) having
		cardinality \(r\).
	\end{definition}
	
	By the convention for an empty intersection, \(q_0(G)=1\), while
	Proposition~\ref{prop:background} gives \(q_1(G)=n\).
	
	\begin{proposition}\label{prop:q-formula}
		For \(1\le i\le n-1\),
		\begin{equation}\label{eq:q-formula}
			\beta_{i,i+1}(\Bbbk[\NN(G)])
			=
			-\binom n{i+1}
			+\sum_{r=1}^{i+1}
			(-1)^{r+1}q_r(G)
			\binom{n-r}{i+1-r}.
		\end{equation}
		Together with \(\beta_{0,0}=1\) and the vanishing outside the linear strand,
		this computes the entire graded Betti table.
	\end{proposition}
	
	\begin{proof}
		Let
		\[
		Q_G(z)=\sum_{r=0}^{n}q_r(G)z^r
		\]
		be the face enumerator of \(\NN(G)\). The face form of the Stanley--Reisner
		Hilbert series gives
		\[
		\operatorname{Hilb}_{\Bbbk[\NN(G)]}(t)
		=Q_G\left(\frac{t}{1-t}\right).
		\]
		Because the resolution is linear, the Hilbert numerator is
		\[
		(1-t)^n\operatorname{Hilb}_{\Bbbk[\NN(G)]}(t)
		=1+\sum_{i=1}^{n-1}
		(-1)^i\beta_{i,i+1}(\Bbbk[\NN(G)])t^{i+1}.
		\]
		
		Also,
		\[
		(1-t)^nQ_G\left(\frac{t}{1-t}\right)
		=\sum_{r=0}^{n}q_r(G)t^r(1-t)^{n-r}.
		\]
		The coefficient of \(t^{i+1}\) in this expression is
		\[
		\sum_{r=0}^{i+1}
		(-1)^{i+1-r}q_r(G)\binom{n-r}{i+1-r}.
		\]
		Equating this coefficient with
		\((-1)^i\beta_{i,i+1}(\Bbbk[\NN(G)])\), multiplying by \((-1)^i\), and
		using \(q_0(G)=1\) gives \eqref{eq:q-formula}.
	\end{proof}

	\begin{corollary}\label{cor:q-recovery}
		For a graph satisfying Proposition~\ref{prop:background}, the graded Betti
		table and \(n\) determine the complete face vector
		\((q_0(G),\ldots,q_n(G))\). Explicitly, \(q_0=1\), \(q_1=n\), and for
		\(2\le k\le n\),
		\begin{equation}\label{eq:q-inversion}
		q_k(G)=(-1)^{k+1}\left[
		\beta_{k-1,k}+\binom nk
		-\sum_{r=1}^{k-1}(-1)^{r+1}q_r(G)
		\binom{n-r}{k-r}
		\right].
		\end{equation}
	\end{corollary}

	\begin{proof}
		Equation~\eqref{eq:q-formula} with \(i=k-1\) has coefficient
		\((-1)^{k+1}\) on \(q_k(G)\). The displayed recursion follows by isolating
		this term.
	\end{proof}
	
	In particular,
	\[
	\beta_{0,2}(I_{\NN(G)})
	=\beta_{1,2}(\Bbbk[\NN(G)])
	=\binom n2-q_2(G),
	\]
	the number of unordered vertex pairs having no common neighbor.
	
	The terminal entry of the linear strand controls the projective dimension and
	depth. The same Betti formulas also determine the multiplicity and the relevant
	Cohen--Macaulay properties.

	\begin{theorem}\label{thm:terminal}
		Let \(G\) have \(n\) vertices and \(c\) connected components. If
		\(I_{\NN(G)}\) has a \(2\)-linear resolution, then
		\[
		\beta_{n-1,n}(\Bbbk[\NN(G)])=2c-1.
		\]
		Consequently,
		\begin{equation}\label{eq:pd-depth}
			\pd\Bbbk[\NN(G)]=n-1,\qquad
			\depth\Bbbk[\NN(G)]=1.
		\end{equation}
		Moreover,
		\begin{equation}\label{eq:dimension}
			\dim\Bbbk[\NN(G)]=d(G),
		\end{equation}
		and \(\Bbbk[\NN(G)]\) is initially Cohen--Macaulay if and only if
		\(G\) has a connected component isomorphic to \(K_{1,m}\) for some
		\(m\ge1\). Consequently, \(\Bbbk[\NN(G)]\) is Cohen--Macaulay if and
		only if \(G\cong cK_2\).
	\end{theorem}
	
	\begin{proof}
		By Theorem~\ref{thm:structural}, the graph \(G\) is bipartite. Let
		\(V=X\sqcup Y\) be a bipartition, and let \(C\) be a connected component of
		\(G\). Any two vertices in \(X\cap V(C)\) are joined by an even path in
		\(G\), which induces a walk in the half square
		\(\CN(G)[X\cap V(C)]\). Hence this half square is connected. The same
		argument applies to \(Y\cap V(C)\). No edge of the common neighbor graph joins different components of \(G\),
		and no edge joins the two color classes. Each
		component of \(G\) therefore contributes exactly two components to
		\(\CN(G)\), so
		\[
		c(\CN(G))=2c.
		\]
		Taking \(W=V\) in \eqref{eq:multigraded} gives
		\[
		\beta_{n-1,n}(\Bbbk[\NN(G)])=c(\CN(G))-1=2c-1.
		\]
		This Betti number is nonzero in homological degree \(n-1\), and
		Theorem~\ref{thm:all-betti} shows that no larger homological degree occurs.
		Thus
		\[
		\pd\Bbbk[\NN(G)]=n-1.
		\]
		The Auslander--Buchsbaum formula then yields
		\(\depth\Bbbk[\NN(G)]=1\).
		
		For the dimension, every face of \(\NN(G)\) is contained in an
		open neighborhood, and every open neighborhood is itself a face. The maximum
		cardinality of a face is therefore \(d(G)\), proving
		\[
		\dim\Bbbk[\NN(G)]=d(G).
		\]
		
		Since the depth is one, \(\Bbbk[\NN(G)]\) is initially Cohen--Macaulay if
		and only if \(\NN(G)\) has a facet of dimension zero. Suppose first that
		\(\{u\}\) is such a facet. Proposition~\ref{prop:background} gives
		\(N_G(u)\ne\varnothing\). If a vertex \(y\in N_G(u)\) had a neighbor
		\(v\ne u\), then \(y\) would be a common neighbor of \(u\) and \(v\), so
		\[
		\{u,v\}\in\NN(G),
		\]
		contrary to the maximality of \(\{u\}\). Thus every neighbor of \(u\) has
		degree one, and the component containing \(u\) is
		\(K_{1,\deg_G(u)}\).
		
		Conversely, suppose that a component of \(G\) is a star \(K_{1,m}\), and let
		\(u\) be its center. The singleton \(\{u\}\) is a face of \(\NN(G)\). For
		every \(v\ne u\), the vertices \(u\) and \(v\) have no common neighbor. If
		\(v\) lies in another component, their neighborhoods are disjoint; if \(v\)
		is a leaf of the star, then
		\[
		N_G(u)\cap N_G(v)=\varnothing.
		\]
		No face therefore properly contains \(\{u\}\), so \(\{u\}\) is a facet.
		This proves the initially Cohen--Macaulay characterization.
		
		Finally, Cohen--Macaulayness is equivalent to
		\(\depth\Bbbk[\NN(G)]=\dim\Bbbk[\NN(G)]\), and hence to \(d(G)=1\).
		Proposition~\ref{prop:background} excludes isolated vertices, so \(d(G)=1\)
		holds exactly when every component of \(G\) is \(K_2\). Therefore
		\(G\cong cK_2\).
	\end{proof}
	
	Thus the Betti table determines the number of connected components by
	\[
	c(G)=\frac{\beta_{n-1,n}+1}{2},
	\]
	and every terminal entry is odd.
	
	\begin{proposition}\label{prop:multiplicity}
		Let \(d=d(G)\), and count equal open neighborhoods only once.
		Then
		\[
		\mult\bigl(\Bbbk[\NN(G)]\bigr)
		=
		\left|
		\left\{
		N_G(v):v\in V,\ \deg_G(v)=d
		\right\}
		\right|.
		\]
	\end{proposition}
	
	\begin{proof}
		The associativity formula for a Stanley--Reisner ring identifies its
		multiplicity with the number of facets of maximum cardinality. The minimal
		primes of maximum dimension correspond to these facets, and each localized
		quotient has multiplicity one.
		
		Let \(A\) be a face of \(\NN(G)\) with \(|A|=d\). By definition of the
		neighborhood complex, there exists \(v\in V\) such that
		\(A\subseteq N_G(v)\). Since \(|N_G(v)|\le d\), we must have
		\[
		A=N_G(v)
		\qquad\text{and}\qquad
		\deg_G(v)=d.
		\]
		Conversely, if \(\deg_G(v)=d\), then \(N_G(v)\) is a face of cardinality
		\(d\), and hence a facet of maximum dimension. Thus the facets of maximum dimension are exactly the distinct open
		neighborhoods of vertices of maximum degree. Counting each such set once gives
		the formula.
	\end{proof}
	
	\section{Degree data and overlap corrections}\label{sec:overlap}
	
	Throughout this section, assume that \(I_{\NN(G)}\) has a \(2\)-linear
	resolution. The face numbers record whether a set has a common neighbor, but not
	how many common neighbors it has. The following invariant measures the resulting
	overcount.
	
	\begin{definition}
		For \(A\subseteq V\), put
		\[
		\mu_G(A)=\left|\bigcap_{a\in A}N_G(a)\right|.
		\]
		For \(r\ge2\), define the \emph{common neighbor overlap of order \(r\)}
		\[
		\varepsilon_r(G)=
		\sum_{\substack{A\subseteq V\\|A|=r\\\mu_G(A)>0}}
		\bigl(\mu_G(A)-1\bigr).
		\]
	\end{definition}
	
	\begin{proposition}\label{prop:overlap-formula}
		Assume that \(I_{\NN(G)}\) has a \(2\)-linear resolution. Let
		\(d_v=\deg_G(v)\) and \(D=\sum_{v\in V}d_v\). For
		\(1\le i\le n-1\),
		\begin{align}
			\beta_{i,i+1}(\Bbbk[\NN(G)])
			={}&(n-1)\binom n{i+1}
			+(n-D)\binom{n-1}{i}
			-\sum_{v\in V}\binom{n-d_v}{i+1}\notag\\
			&+\sum_{r=2}^{i+1}
			(-1)^r\varepsilon_r(G)
			\binom{n-r}{i+1-r}.
			\label{eq:overlap-formula}
		\end{align}
	\end{proposition}
	
	\begin{proof}
		For \(r\ge2\), double counting the pairs \((A,v)\) with \(|A|=r\) and
		\(A\subseteq N_G(v)\), first by \(v\) and then by \(A\), gives
		\[
		\sum_{v\in V}\binom{d_v}{r}
		=
		\sum_{\substack{A\subseteq V\\|A|=r}}\mu_G(A)
		=q_r(G)+\varepsilon_r(G).
		\]
		Consequently, the face enumerator satisfies
		\begin{align*}
			Q_G(z)
			&=1+nz+
			\sum_{r\ge2}
			\left(\sum_{v\in V}\binom{d_v}{r}-\varepsilon_r(G)\right)z^r\\
			&=1-n+\sum_{v\in V}(1+z)^{d_v}
			+(n-D)z-\sum_{r\ge2}\varepsilon_r(G)z^r.
		\end{align*}
		
		Substituting \(z=t/(1-t)\) and multiplying by \((1-t)^n\) gives
		\begin{align*}
			(1-t)^nQ_G\left(\frac{t}{1-t}\right)
			={}&(1-n)(1-t)^n
			+\sum_{v\in V}(1-t)^{n-d_v}\\
			&+(n-D)t(1-t)^{n-1}
			-\sum_{r\ge2}\varepsilon_r(G)t^r(1-t)^{n-r}.
		\end{align*}
		The coefficient of \(t^{i+1}\) is
		\((-1)^i\beta_{i,i+1}(\Bbbk[\NN(G)])\). Extracting this coefficient and
		multiplying by \((-1)^i\) yields
		\begin{align*}
			\beta_{i,i+1}(\Bbbk[\NN(G)])
			={}&(n-1)\binom n{i+1}
			+(n-D)\binom{n-1}{i}
			-\sum_{v\in V}\binom{n-d_v}{i+1}\\
			&+\sum_{r=2}^{i+1}
			(-1)^r\varepsilon_r(G)\binom{n-r}{i+1-r},
		\end{align*}
		which is \eqref{eq:overlap-formula}.
	\end{proof}

	Here a graph without \(K_{2,3}\) means that it contains no subgraph
	isomorphic to \(K_{2,3}\); the subgraph need not be induced.

	\begin{corollary}\label{cor:k23-free-formula}
		Suppose that \(I_{\NN(G)}\) has a \(2\)-linear resolution and that \(G\)
		contains no \(K_{2,3}\) as a subgraph.
		Let \(c_4(G)\) denote the number of \(4\)-cycles of \(G\). Then
		\[
		\varepsilon_2(G)=2c_4(G),
		\qquad
		\varepsilon_r(G)=0\quad\text{for every }r\ge3.
		\]
		Consequently, for \(1\le i\le n-1\),
		\begin{align}
		\beta_{i,i+1}(\Bbbk[\NN(G)])
		={}&(n-1)\binom n{i+1}
		+(n-D)\binom{n-1}{i}
		-\sum_{v\in V}\binom{n-d_v}{i+1}\notag\\
		&+2c_4(G)\binom{n-2}{i-1},
		\label{eq:k23-free-formula}
		\end{align}
		where \(d_v=\deg_G(v)\) and \(D=\sum_{v\in V}d_v\).
	\end{corollary}

	\begin{proof}
		If a set of at least three vertices had two common neighbors, then any three
		of its vertices together with two such neighbors would form a \(K_{2,3}\).
		Hence \(\varepsilon_r(G)=0\) for every \(r\ge3\).

		By Theorem~\ref{thm:structural}, \(G\) is bipartite. Since \(G\) also
		contains no \(K_{2,3}\), two vertices have at most two
		common neighbors. A pair has exactly two common neighbors precisely when the
		pair and those neighbors form a \(4\)-cycle. Each \(4\)-cycle contributes its
		two pairs of opposite vertices, and each such pair contributes one to
		\(\varepsilon_2(G)\). Thus \(\varepsilon_2(G)=2c_4(G)\). Formula
		\eqref{eq:k23-free-formula} now follows from
		Proposition~\ref{prop:overlap-formula}.
	\end{proof}

	\begin{example}\label{ex:degree-counterexample}
		Let \(X=\{x_1,x_2,x_3\}\) and
		\(Y=\{y_1,y_2,y_3,y_4\}\). Define two bipartite graphs by
		\[
		\begin{array}{c|ccc}
			&N(x_1)&N(x_2)&N(x_3)\\ \hline
			G_A&
			\{y_1,y_3,y_4\}&
			\{y_1,y_2,y_3\}&
			\{y_1,y_2\}\\[2pt]
			G_B&
			\{y_1,y_3,y_4\}&
			\{y_1,y_3,y_4\}&
			\{y_1,y_2\}.
		\end{array}
		\]
		The graphs are shown in Figure~\ref{fig:degree-counterexample}.
		
		\begin{figure}[ht]
			\centering
\begin{tikzpicture}[
				scale=0.82,
				xvertex/.style={circle,draw,minimum size=6.5mm,inner sep=1pt},
				yvertex/.style={rectangle,rounded corners=1.5pt,draw,
					minimum size=6.5mm,inner sep=1pt},
				graphedge/.style={thick}
				]
				
				\begin{scope}[xshift=-3.6cm]
					\node at (1.45,2.45) {$G_A$};
					\node[xvertex] (Ax1) at (0,1.5) {$x_1$};
					\node[xvertex] (Ax2) at (0,0) {$x_2$};
					\node[xvertex] (Ax3) at (0,-1.5) {$x_3$};
					\node[yvertex] (Ay4) at (2.9,1.8) {$y_4$};
					\node[yvertex] (Ay3) at (2.9,0.6) {$y_3$};
					\node[yvertex] (Ay1) at (2.9,-0.6) {$y_1$};
					\node[yvertex] (Ay2) at (2.9,-1.8) {$y_2$};
					\draw[graphedge] (Ax1)--(Ay1) (Ax1)--(Ay3) (Ax1)--(Ay4);
					\draw[graphedge] (Ax2)--(Ay1) (Ax2)--(Ay2) (Ax2)--(Ay3);
					\draw[graphedge] (Ax3)--(Ay1) (Ax3)--(Ay2);
				\end{scope}
				
				\begin{scope}[xshift=3.6cm]
					\node at (1.45,2.45) {$G_B$};
					\node[xvertex] (Bx1) at (0,1.5) {$x_1$};
					\node[xvertex] (Bx2) at (0,0) {$x_2$};
					\node[xvertex] (Bx3) at (0,-1.5) {$x_3$};
					\node[yvertex] (By4) at (2.9,1.8) {$y_4$};
					\node[yvertex] (By3) at (2.9,0.6) {$y_3$};
					\node[yvertex] (By1) at (2.9,-0.6) {$y_1$};
					\node[yvertex] (By2) at (2.9,-1.8) {$y_2$};
					\draw[graphedge] (Bx1)--(By1) (Bx1)--(By3) (Bx1)--(By4);
					\draw[graphedge] (Bx2)--(By1) (Bx2)--(By3) (Bx2)--(By4);
					\draw[graphedge] (Bx3)--(By1) (Bx3)--(By2);
				\end{scope}
			\end{tikzpicture}
			\caption{Two graphs with the same degree sequences but different overlap data}
			\label{fig:degree-counterexample}
		\end{figure}
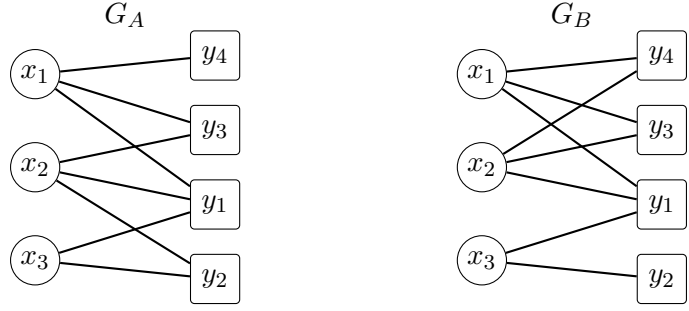
		
		Both graphs are connected and chordal bipartite. In \(G_A\), the vertex
		\(y_4\) is pendant, and the only possible \(6\)-cycle on the remaining
		vertices has a chord through \(x_2\). In \(G_B\), the vertex \(y_2\) is
		pendant, and every cycle has length at most four. Hence
		Corollary~\ref{cor:chordal-bipartite} applies.
		
		The graphs have the same degree sequences on both sides:
		\[
		(d(x_1),d(x_2),d(x_3))=(3,3,2),\qquad
		\{d(y):y\in Y\}=\{3,2,2,1\}.
		\]
		However,
		\[
		\begin{array}{c|ccc|cc}
			&q_1&q_2&q_3&\varepsilon_2&\varepsilon_3\\ \hline
			G_A&7&8&3&4&0\\
			G_B&7&7&2&5&1
		\end{array}
		\]
		and \(q_r=0\) for \(r\ge4\). Proposition~\ref{prop:q-formula} therefore gives
		\[
		\bigl(\beta_{i,i+1}(\Bbbk[\NN(G_A)])\bigr)_{i=1}^{6}
		=(13,33,37,22,7,1)
		\]
		and
		\[
		\bigl(\beta_{i,i+1}(\Bbbk[\NN(G_B)])\bigr)_{i=1}^{6}
		=(14,37,43,26,8,1).
		\]
		Thus neither \(n\), \(c\), nor the degree sequences of the two color
		classes determine the Betti table, even within the class of chordal bipartite
		graphs.
	\end{example}
	
	\section{Cactus graphs, pseudoforests, and forest rigidity}\label{sec:pseudoforests}

	Theorem~\ref{thm:cactus-characterization} identifies the cactus graphs in the
	\(2\)-linear class. Such graphs contain no \(K_{2,3}\), so
	Corollary~\ref{cor:k23-free-formula} reduces their overlap corrections to the
	number of \(4\)-cycles. The cactus block structure then gives a closed formula
	and shows how much degree data the Betti table retains. Pseudoforests and
	forests are successive specializations.

	\subsection{The cactus formula}

	For a cactus graph \(G\), let \(s(G)\) denote the number of its cycles.

	\begin{lemma}\label{lem:cactus-overlaps}
		Let \(G\) be a cactus graph in which every cycle has length four. Then
		\[
		\varepsilon_2(G)=2s(G),
		\qquad
		\varepsilon_r(G)=0\quad\text{for every }r\ge3.
		\]
		If \(G\) has \(n\) vertices and \(c\) connected components, then
		\[
		|E(G)|=n-c+s(G).
		\]
	\end{lemma}

	\begin{proof}
		Suppose that two vertices have two distinct common neighbors. These four
		vertices form a \(4\)-cycle. They cannot have a third common neighbor, since
		two of the resulting \(4\)-cycles would share more than one vertex. Thus each
		pair contributing to \(\varepsilon_2(G)\) contributes exactly one.

		Each \(4\)-cycle gives two such pairs, namely its two pairs of opposite
		vertices. Conversely, every pair with two common neighbors determines a
		unique \(4\)-cycle. Hence \(\varepsilon_2(G)=2s(G)\).

		If a set of at least three vertices had two common neighbors, then two
		different pairs in that set, together with the common neighbors, would produce
		two \(4\)-cycles sharing more than one vertex. Therefore
		\(\varepsilon_r(G)=0\) for every \(r\ge3\).

		Finally, delete one edge from each cycle. Every deletion preserves the number
		of connected components, and the resulting graph is a forest. It therefore
		has \(n-c\) edges, which gives \(|E(G)|-s(G)=n-c\).
	\end{proof}

	\begin{theorem}\label{thm:cactus-formula}
		Let \(G\) be a cactus graph without isolated vertices, and suppose that every
		cycle of \(G\) has length four. Set
		\[
		n=|V(G)|,\qquad c=c(G),\qquad s=s(G),\qquad d_v=\deg_G(v).
		\]
		Apart from \(\beta_{0,0}=1\), the only nonzero graded Betti numbers of
		\(\Bbbk[\NN(G)]\) are \(\beta_{i,i+1}\), \(1\le i\le n-1\), and
		\begin{equation}\label{eq:cactus-formula}
		\begin{aligned}
		\beta_{i,i+1}(\Bbbk[\NN(G)])
		={}&(n-1)\binom{n}{i+1}
		+(2c-n)\binom{n-1}{i}\\
		&-\sum_{v\in V(G)}\binom{n-d_v}{i+1}
		-2s\binom{n-2}{i}.
		\end{aligned}
		\end{equation}
		Equivalently, if \(D=\sum_{v\in V(G)}d_v\), then
		\begin{equation}\label{eq:cactus-degree-formula}
		\begin{aligned}
		\beta_{i,i+1}(\Bbbk[\NN(G)])
		={}&(n-1)\binom{n}{i+1}
		+(2c-n)\binom{n-1}{i}\\
		&+(2n-2c-D)\binom{n-2}{i}
		-\sum_{v\in V(G)}\binom{n-d_v}{i+1}.
		\end{aligned}
		\end{equation}
	\end{theorem}

	\begin{proof}
		Theorem~\ref{thm:cactus-characterization} gives a \(2\)-linear resolution.
		By Lemma~\ref{lem:cactus-overlaps},
		\[
		D=2|E(G)|=2(n-c+s),\qquad
		\varepsilon_2(G)=2s,
		\qquad
		\varepsilon_r(G)=0\quad(r\ge3).
		\]
		Substitution in \eqref{eq:overlap-formula} gives
		\begin{align*}
		\beta_{i,i+1}
		={}&(n-1)\binom{n}{i+1}
		+(2c-n-2s)\binom{n-1}{i}\\
		&-\sum_{v\in V(G)}\binom{n-d_v}{i+1}
		+2s\binom{n-2}{i-1}.
		\end{align*}
		The identity
		\[
		\binom{n-1}{i}-\binom{n-2}{i-1}=\binom{n-2}{i}
		\]
		yields \eqref{eq:cactus-formula}. Since
		\(2s=D-2n+2c\), formula \eqref{eq:cactus-degree-formula} follows.
	\end{proof}

	A graph is a \emph{pseudoforest} if each connected component contains at most
	one cycle. Thus every pseudoforest is a cactus graph, and forests are precisely
	the acyclic pseudoforests. The cactus formula has the following immediate
	specialization.

	\begin{corollary}\label{cor:pseudoforest}
		Let \(P\) be a pseudoforest without isolated vertices. Then
		\(I_{\NN(P)}\) has a \(2\)-linear resolution if and only if every cycle of
		\(P\) has length four.

		Assume these equivalent conditions. Let \(n=|V(P)|\), let \(c=c(P)\), let
		\(u\) be the number of unicyclic components, and write
		\(d_v=\deg_P(v)\). Apart from \(\beta_{0,0}=1\), the only nonzero graded
		Betti numbers are \(\beta_{i,i+1}\), \(1\le i\le n-1\), and
		\begin{equation}\label{eq:pseudoforest-formula}
		\begin{aligned}
		\beta_{i,i+1}(\Bbbk[\NN(P)])
		={}&(n-1)\binom{n}{i+1}
		+(2c-n)\binom{n-1}{i}\\
		&-\sum_{v\in V(P)}\binom{n-d_v}{i+1}
		-2u\binom{n-2}{i}.
		\end{aligned}
		\end{equation}
	\end{corollary}

	\begin{proof}
		The characterization follows from Theorem~\ref{thm:cactus-characterization}.
		Since a pseudoforest has \(s(P)=u\), the formula follows from
		Theorem~\ref{thm:cactus-formula}.
	\end{proof}

	\subsection{Degree information for cactus graphs}

	For \(d\ge1\), write
	\[
	h_d(G)=\bigl|\{v\in V(G):\deg_G(v)=d\}\bigr|.
	\]
	The cactus formula does not retain the full rigidity of the forest case. The
	following result identifies exactly which degree data remain visible.

	\begin{theorem}\label{thm:cactus-rigidity}
		Let \(G\) and \(G'\) be cactus graphs without isolated vertices whose cycles
		all have length four. The Betti tables of their neighborhood complexes agree
		if and
		only if
		\[
		|V(G)|=|V(G')|,\qquad c(G)=c(G'),
		\]
		and
		\[
		h_d(G)=h_d(G')\qquad\text{for every }d\ge3.
		\]
	\end{theorem}

	\begin{proof}
		For \(r\ge3\), Lemma~\ref{lem:cactus-overlaps} gives
		\[
		q_r(G)=\sum_{d\ge r}h_d(G)\binom{d}{r}.
		\]
		For \(r=2\), the same lemma and the identity
		\(2s=\sum_v d_v-2n+2c\) give
		\begin{equation}\label{eq:cactus-q2}
		q_2(G)
		=n-2c+\sum_{d\ge3}h_d(G)\binom{d-1}{2}.
		\end{equation}
		Thus \(n\), \(c\), and the numbers \(h_d(G)\) for \(d\ge3\) determine all
		the face numbers \(q_r(G)\). Proposition~\ref{prop:q-formula} then determines
		the Betti table.

		Conversely, suppose that the Betti tables agree. By
		Theorem~\ref{thm:terminal}, the terminal homological degree determines \(n\),
		and the terminal entry determines
		\[
		c=\frac{\beta_{n-1,n}+1}{2}.
		\]
		Formula~\eqref{eq:q-formula} is triangular in the values \(q_r(G)\), so the
		Betti table determines every \(q_r(G)\). For \(d\ge3\), binomial inversion
		then gives
		\begin{equation}\label{eq:cactus-degree-inversion}
		h_d(G)=
		\sum_{r=d}^{n-1}
		(-1)^{r-d}\binom{r}{d}q_r(G).
		\end{equation}
		Hence the table determines \(h_d(G)\) for every \(d\ge3\), proving the
		converse.
	\end{proof}

	\begin{corollary}\label{cor:cactus-low-degrees}
		Let \(G\) be a cactus graph without isolated vertices whose cycles all have
		length four, and write
		\[
		n=|V(G)|,\qquad c=c(G),\qquad s=s(G).
		\]
		Then
		\begin{align}
		h_1(G)
		&=2c+\sum_{d\ge3}(d-2)h_d(G)-2s,
		\label{eq:cactus-h1}\\
		h_2(G)
		&=n-2c-\sum_{d\ge3}(d-1)h_d(G)+2s.
		\label{eq:cactus-h2}
		\end{align}
		Consequently, the Betti table together with \(s(G)\) determines the complete
		degree sequence. If \(G'\) is another such cactus graph with the same Betti
		table, then
		\[
		h_1(G)-h_1(G')=-2\bigl(s(G)-s(G')\bigr)
		\]
		and
		\[
		h_2(G)-h_2(G')=2\bigl(s(G)-s(G')\bigr).
		\]
	\end{corollary}

	\begin{proof}
		The identities
		\[
		h_1+h_2+\sum_{d\ge3}h_d=n
		\]
		and
		\[
		h_1+2h_2+\sum_{d\ge3}d h_d
		=\sum_{v\in V(G)}d_v
		=2(n-c+s)
		\]
		form a linear system in \(h_1\) and \(h_2\). Solving it gives
		\eqref{eq:cactus-h1} and \eqref{eq:cactus-h2}. Theorem~\ref{thm:cactus-rigidity} shows that the Betti table determines
		\(n\), \(c\),
		and every \(h_d\) with \(d\ge3\), so the first consequence follows. The two
		difference formulas follow by comparing \eqref{eq:cactus-h1} and
		\eqref{eq:cactus-h2} for \(G\) and \(G'\).
	\end{proof}

	\begin{example}\label{ex:path-square-same-betti}
		The graphs \(P_4\) and \(C_4\) have degree sequences
		\[
		(2,2,1,1)
		\qquad\text{and}\qquad
		(2,2,2,2),
		\]
		respectively. Nevertheless, both are connected cactus graphs on four vertices
		with no vertex of degree at least three. Theorem~\ref{thm:cactus-rigidity}
		therefore gives the same Betti strand:
		\[
		\bigl(\beta_{1,2},\beta_{2,3},\beta_{3,4}\bigr)=(4,4,1).
		\]
		Here \(s(P_4)=0\) and \(s(C_4)=1\). Corollary~\ref{cor:cactus-low-degrees}
		explains the difference between their degree sequences: passing from \(P_4\)
		to \(C_4\) decreases \(h_1\) by two and increases \(h_2\) by two without
		changing the Betti table.
	\end{example}

	\subsection{The forest case}

	The cactus formula contains the forest formula as the case \(s=0\). The
	following proposition explains structurally why the overlap terms disappear
	exactly in the acyclic case.

	\begin{proposition}\label{prop:forest-zero-overlap}
		Let \(G\) be a chordal bipartite graph. The following are equivalent:
		\begin{enumerate}[label=\textup{(\roman*)}]
			\item \(G\) is a forest;
			\item \(\varepsilon_2(G)=0\);
			\item \(\varepsilon_r(G)=0\) for every \(r\ge2\).
		\end{enumerate}
	\end{proposition}

	\begin{proof}
		Conditions \textup{(ii)} and \textup{(iii)} are equivalent. The equality
		\(\varepsilon_2(G)=0\) means that no pair of vertices has two distinct common
		neighbors. In a bipartite graph, two vertices with two distinct common
		neighbors, together with those neighbors, induce a \(C_4\). Hence
		\[
		\varepsilon_2(G)=0
		\quad\Longleftrightarrow\quad
		G\text{ is }C_4\text{-free}.
		\]
		If \(\varepsilon_r(G)>0\) for some \(r\ge2\), then an \(r\)-set has at least
		two common neighbors. Choosing two vertices of that set and two of its common
		neighbors again produces a \(C_4\). Therefore
		\(\varepsilon_r(G)=0\) for every \(r\ge2\) if and only if
		\(\varepsilon_2(G)=0\).

		This condition is also equivalent to \textup{(i)}. If a chordal
		bipartite graph contains a cycle, then a shortest cycle is induced. Chordal
		bipartiteness forces that induced cycle to have length four. Thus \(G\)
		contains no induced \(C_4\) exactly when it contains no cycle, or equivalently,
		exactly when \(G\) is a forest.
	\end{proof}

	\begin{corollary}\label{cor:forest-formula}
		Let \(F\) be a forest without isolated vertices. Set
		\[
		n=|V(F)|,\qquad c=c(F),\qquad d_v=\deg_F(v).
		\]
		Then, apart from \(\beta_{0,0}=1\), the only nonzero graded Betti numbers of
		\(\Bbbk[\NN(F)]\) are \(\beta_{i,i+1}\), \(1\le i\le n-1\), and
		\begin{equation}\label{eq:forest-formula}
			\beta_{i,i+1}(\Bbbk[\NN(F)])
			=
			(n-1)\binom{n}{i+1}
			+(2c-n)\binom{n-1}{i}
			-\sum_{v\in V(F)}\binom{n-d_v}{i+1}.
		\end{equation}
	\end{corollary}

	\begin{proof}
		A forest is a cactus graph with \(s=0\), so the formula follows immediately
		from Theorem~\ref{thm:cactus-formula}.
	\end{proof}

	For the two basic forest families, formula~\eqref{eq:forest-formula} gives
	\[
	\beta_{i,i+1}(\Bbbk[\NN(cK_2)])=i\binom{2c}{i+1}
	\]
	and
	\[
	\beta_{i,i+1}(\Bbbk[\NN(K_{1,m})])=\binom mi.
	\]
	
	Formula~\eqref{eq:forest-formula} has a converse: it loses no information
	about the degree multiset.

	\begin{theorem}\label{thm:forest-rigidity}
		Let \(F\) and \(F'\) be forests without isolated vertices. The Betti tables
		of their neighborhood complexes agree if and only if \(F\) and \(F'\)
		have the same degree sequence.
	\end{theorem}

	\begin{proof}
		Suppose first that \(F\) and \(F'\) have the same degree sequence. They then
		have the same number \(n\) of vertices and the same number of components,
		because a forest satisfies
		\[
		c=n-\frac12\sum_{v\in V(F)}d_v.
		\]
		Corollary~\ref{cor:forest-formula} therefore gives identical Betti tables.

		Conversely, suppose that the Betti tables agree. By
		Theorem~\ref{thm:terminal}, the terminal homological degree is \(n-1\), so the
		table determines \(n\). The terminal entry also determines the number of
		components:
		\[
		c=\frac{\beta_{n-1,n}+1}{2}.
		\]
		For \(0\le k\le n\), define
		\[
		M_k(F)=\sum_{v\in V(F)}\binom{n-d_v}{k}.
		\]
		The first two values are determined by \(n\) and \(c\):
		\begin{equation}\label{eq:first-moments}
			M_0(F)=n,
			\qquad
			M_1(F)=n^2-2(n-c).
		\end{equation}
		For \(2\le k\le n\), Corollary~\ref{cor:forest-formula}, with \(i=k-1\),
		gives
		\begin{equation}\label{eq:moment-from-betti}
			M_k(F)
			=
			(n-1)\binom nk
			+(2c-n)\binom{n-1}{k-1}
			-\beta_{k-1,k}(\Bbbk[\NN(F)]).
		\end{equation}
		Thus the Betti table determines every \(M_k(F)\).

		Now set
		\[
		h_j(F)=\bigl|\{v\in V(F):n-d_v=j\}\bigr|,
		\qquad 0\le j\le n.
		\]
		Grouping the vertices according to the value of \(n-d_v\) gives
		\[
		M_k(F)=\sum_{j=k}^{n}h_j(F)\binom jk.
		\]
		Binomial inversion yields
		\begin{equation}\label{eq:binomial-inversion}
			h_j(F)=
			\sum_{k=j}^{n}
			(-1)^{k-j}\binom kj M_k(F),
			\qquad 0\le j\le n.
		\end{equation}
		Hence the table determines every \(h_j(F)\). Since the number of vertices of
		degree \(d\) is \(h_{n-d}(F)\), it determines the full degree sequence.
	\end{proof}

	\begin{remark}
		The proof of Theorem~\ref{thm:forest-rigidity} is constructive: the terminal
		entry gives \(n\) and \(c\), equations~\eqref{eq:first-moments} and
		\eqref{eq:moment-from-betti} give the moments \(M_k\), and
		\eqref{eq:binomial-inversion} recovers the multiplicity of each degree.
	\end{remark}

	\section{Dominance complexes}\label{sec:dominance}
	
	This final section applies the preceding viewpoint to dominance complexes and
	closed neighborhood ideals.
	
	For a finite simple graph \(G\), its \emph{dominance complex} is
	\[
	\DD(G)
	=
	\left\{
	F\subseteq V(G):
	V(G)\setminus F
	\text{ is a dominating set of }G
	\right\}.
	\]
	A set \(F\subseteq V(G)\) is a nonface of \(\DD(G)\) if and only if
	\(V(G)\setminus F\) fails to dominate some vertex \(v\). Equivalently,
	\[
	N_G[v]\cap\bigl(V(G)\setminus F\bigr)=\varnothing
	\]
	for some \(v\in V(G)\), or, equivalently,
	\[
	N_G[v]\subseteq F.
	\]
	Therefore
	\begin{equation}\label{eq:dominance-closed-neighborhood}
		I_{\DD(G)}
		=
		\bigl(x_{N_G[v]}:v\in V(G)\bigr)
		=
		NI(G),
	\end{equation}
	where \(NI(G)\) is the closed neighborhood ideal introduced by Sharifan
	and Moradi \cite[Definition~2.1]{SharifanMoradi2020}. This generating set need
	not be minimal. Its minimal generators correspond to the members that are minimal
	under inclusion in
	\[
	\{N_G[v]:v\in V(G)\}.
	\]
	
	The same description also yields the relation with the neighborhood complex
	of the complement. Recall that the Alexander dual of a simplicial complex
	\(\Delta\) on \(V\) is
	\[
	\Delta^\vee=\{F\subseteq V:V\setminus F\notin\Delta\}.
	\]
	For every \(v\in V(G)\),
	\[
	N_{\overline G}(v)=V(G)\setminus N_G[v].
	\]
	Therefore, for every \(F\subseteq V(G)\),
	\begin{align*}
		F\in\NN(\overline G)^\vee
		&\Longleftrightarrow
		V(G)\setminus F\notin\NN(\overline G)\\
		&\Longleftrightarrow
		\nexists\,v\in V(G)
		\text{ such that }
		V(G)\setminus F\subseteq N_{\overline G}(v)\\
		&\Longleftrightarrow
		\nexists\,v\in V(G)
		\text{ such that }
		N_G[v]\subseteq F\\
		&\Longleftrightarrow
		F\in\DD(G).
	\end{align*}
	Consequently,
	\[
	\DD(G)=\NN(\overline G)^\vee,
	\qquad\text{equivalently}\qquad
	\DD(G)^\vee=\NN(\overline G).
	\]
	This is precisely the relation from Alexander duality proved in
	\cite[Theorem~1.1]{MatsushitaWakatsuki2025}.
	
	\begin{theorem}\label{thm:dominance-star}
		Let \(G\) be a nonempty finite simple graph. The ideal \(I_{\DD(G)}\)
		has a \(2\)-linear resolution over \(\Bbbk\) if and only if
		\[
		G\cong K_{1,m}
		\]
		for some \(m\ge1\).
	\end{theorem}
	
	\begin{proof}
		Assume first that \(I_{\DD(G)}=NI(G)\) has a \(2\)-linear resolution. Since
		\(NI(G)\) is a nonzero proper ideal, the exact sequence
		\[
		0\longrightarrow NI(G)\longrightarrow S\longrightarrow S/NI(G)
		\longrightarrow0
		\]
		and the linearity assumption give
		\[
		\reg(S/NI(G))=\reg(NI(G))-1=1.
		\]
		The matching number bound in
		\cite[Theorem~1.2]{ChakrabortyJosephRoySingh2025} therefore implies
		\[
		\nu(G)\le1.
		\]
		The graph \(G\) has no isolated vertices. Indeed, if \(v\) were isolated,
		then \(N_G[v]=\{v\}\), and \(x_v\) would be a minimal generator of
		\(NI(G)\) in degree one, contrary to quadratic generation.
		
		The bound \(\nu(G)\le1\) now restricts the structure of \(G\). Since \(G\)
		has no isolated vertices, it contains an edge \(uv\). Since \(\nu(G)\le1\), every edge of
		\(G\) meets \(uv\). Suppose that \(ux\) and \(vy\) are edges with
		\(x,y\notin\{u,v\}\). These edges cannot be disjoint, so \(x=y\), and
		\(u,v,x\) form a triangle.
		
		If such a triangle exists, no fourth vertex can be incident with an edge. An
		edge joining a fourth vertex to one vertex of the triangle would be disjoint
		from the opposite triangle edge, while an edge between two outside vertices
		would be disjoint from \(uv\). Since \(G\) has no isolated vertices, this
		forces \(G\cong K_3\). If no triangle exists, all edges other than \(uv\)
		must be incident with the same endpoint of \(uv\). Otherwise, an edge through
		\(u\) and an edge through \(v\) would either be disjoint or form a triangle.
		The absence of isolated vertices then implies that \(G\) is a star. Thus
		\(G\) is either a star or \(K_3\).
		
		The second possibility is impossible because
		\[
		I_{\DD(K_3)}=NI(K_3)=(x_1x_2x_3),
		\]
		which is generated in degree three. Therefore \(G\cong K_{1,m}\) for some
		\(m\ge1\).
		
		Conversely, let \(G=K_{1,m}\), with center \(c\) and leaves
		\(l_1,\ldots,l_m\). The generator associated with the center is redundant, and
		\eqref{eq:dominance-closed-neighborhood} gives
		\[
		I_{\DD(G)}
		=(x_cx_{l_1},\ldots,x_cx_{l_m})
		=x_c(x_{l_1},\ldots,x_{l_m}).
		\]
		As a graded \(S\)-module, this ideal is isomorphic to
		\((x_{l_1},\ldots,x_{l_m})(-1)\). The ideal generated by the leaf variables
		has a \(1\)-linear Koszul resolution, and multiplication by \(x_c\) shifts
		that resolution by one. Hence \(I_{\DD(G)}\) has a \(2\)-linear resolution.
	\end{proof}


\begingroup
	\small
	\begin{thebibliography}{99}
		
		\bibitem{BrandstadtDraganChepoiVoloshin1998}
		A. Brandst\"adt, F. Dragan, V. Chepoi and V. Voloshin,
		\emph{Dually chordal graphs},
		SIAM J. Discrete Math. \textbf{11} (1998), no.~3, 437--455.
		\href{https://doi.org/10.1137/S0895480193253415}{doi:\nolinkurl{10.1137/S0895480193253415}}.
		
		\bibitem{ChakrabortyJosephRoySingh2025}
		S. Chakraborty, A. P. Joseph, A. Roy and A. Singh,
		\emph{Castelnuovo--Mumford regularity of the closed neighborhood ideal of a graph},
		J. Algebraic Combin. \textbf{61} (2025), article~1.
		\href{https://doi.org/10.1007/s10801-024-01369-0}{doi:\nolinkurl{10.1007/s10801-024-01369-0}}.
		
		\bibitem{EngstromStamps2013}
		A. Engstr\"om and M. T. Stamps,
		\emph{Betti diagrams from graphs},
		Algebr. Number Theory \textbf{7} (2013), no.~7, 1725--1742.
		\href{https://doi.org/10.2140/ant.2013.7.1725}{doi:\nolinkurl{10.2140/ant.2013.7.1725}}.
		
		\bibitem{Froberg1990}
		R. Fr\"oberg,
		\emph{On Stanley--Reisner rings},
		Banach Center Publ. \textbf{26}, Part~2 (1990), 57--70.
		
		\href{https://doi.org/10.4064/-26-2-57-70}{doi:\nolinkurl{10.4064/-26-2-57-70}}.
		
		\bibitem{Froberg2022}
		R. Fr\"oberg,
		\emph{Betti numbers of fat forests and their Alexander dual},
		J. Algebraic Combin. \textbf{56} (2022), 1023--1030.
		\href{https://doi.org/10.1007/s10801-022-01143-0}{doi:\nolinkurl{10.1007/s10801-022-01143-0}}.
		
		\bibitem{Froberg2026}
		R. Fr\"oberg,
		\emph{Stanley--Reisner rings of neighborhood complexes and linear resolutions},
		J. Algebraic Combin. \textbf{63} (2026), article~56.
		\href{https://doi.org/10.1007/s10801-026-01530-x}{doi:\nolinkurl{10.1007/s10801-026-01530-x}}.
		
		\bibitem{GroshausSzwarcfiter2008}
		M. Groshaus and J. L. Szwarcfiter,
		\emph{On hereditary Helly classes of graphs},
		Discrete Math. Theor. Comput. Sci. \textbf{10} (2008), no.~1, 71--78.
		\href{https://doi.org/10.46298/dmtcs.440}{doi:\nolinkurl{10.46298/dmtcs.440}}.
		
		\bibitem{Hochster1977}
		M. Hochster,
		\emph{Cohen--Macaulay rings, combinatorics, and simplicial complexes},
		in \emph{Ring Theory II (Proc. Second Oklahoma Ring Theory Conference, Norman, 1975)},
		Lecture Notes in Pure and Applied Mathematics, vol.~26,
		Marcel Dekker, New York, 1977, 171--223.
		
		\bibitem{LeLe2019}
		H.-O. Le and V. B. Le,
		\emph{Hardness and structural results for half-squares of restricted tree convex bipartite graphs},
		Algorithmica \textbf{81} (2019), 4258--4274.
		\href{https://doi.org/10.1007/s00453-018-0440-7}{doi:\nolinkurl{10.1007/s00453-018-0440-7}}.
		
		\bibitem{Lovasz1978}
		L. Lov\'asz,
		\emph{Kneser's conjecture, chromatic number, and homotopy},
		J. Combin. Theory Ser. A \textbf{25} (1978), no.~3, 319--324.
		\href{https://doi.org/10.1016/0097-3165(78)90022-5}{doi:\nolinkurl{10.1016/0097-3165(78)90022-5}}.
		
		\bibitem{MatsushitaWakatsuki2025}
		T. Matsushita and S. Wakatsuki,
		\emph{Dominance complexes, neighborhood complexes and combinatorial Alexander duals},
		J. Combin. Theory Ser. A \textbf{211} (2025), article~105978.
		\href{https://doi.org/10.1016/j.jcta.2024.105978}{doi:\nolinkurl{10.1016/j.jcta.2024.105978}}.
		
		\bibitem{NamiqInitiallyCM}
		M. R. Namiq,
		\emph{Initially Cohen--Macaulay modules},
		New Math. Nat. Comput., published online (2026).
		\href{https://doi.org/10.1142/S1793005728500226}{doi:\nolinkurl{10.1142/S1793005728500226}}.
		
		\bibitem{Namiq2026}
		M. R. Namiq,
		\emph{Explicit Betti numbers for skeletons of chordal clique complexes and their Alexander duals}, (2026).
		\href{https://arxiv.org/abs/2603.17776}{\nolinkurl{arXiv:2603.17776}}.
		
		\bibitem{SharifanMoradi2020}
		L. Sharifan and S. Moradi,
		\emph{Closed neighborhood ideal of a graph},
		Rocky Mountain J. Math. \textbf{50} (2020), no.~3, 1097--1107.
		\href{https://doi.org/10.1216/rmj.2020.50.1097}{doi:\nolinkurl{10.1216/rmj.2020.50.1097}}.
		
	\end{thebibliography}
	\endgroup
\end{document}